\documentclass[12pt]{article}

\RequirePackage{amsthm,amsmath,amsfonts,amssymb}
\RequirePackage[numbers]{natbib}

\usepackage{amsthm,amsmath,amsfonts,amssymb}
\usepackage[numbers]{natbib}
\usepackage{amsthm,amsmath,natbib}
\usepackage{graphicx}
\usepackage{dcolumn}
\usepackage{bm}
\usepackage{natbib}
\usepackage[utf8]{inputenc}
\usepackage[T1]{fontenc}
\usepackage{mathptmx}
\usepackage{amssymb}
\usepackage{bm}
 \usepackage{setspace} 

\newtheorem{theorem}{Theorem}
\newtheorem{definition}{Definition}

\begin{document}

\title{On optimal linear prediction}

\author{Inge S. Helland\\Department of Mathematics, University of Oslo, \\P.O. Box 1053, 0316 Oslo, Norway\\ingeh@math.uio.no\\Orcid: 0000-0002-7136-873X}

\date{}

\maketitle

\begin{spacing}{1.35}

\begin{abstract}
The main purpose of this article is to prove that, under certain assumptions in a linear prediction setting, optimal methods based upon model reduction and even an optimal predictor can be provided. The optimality is formulated in terms of the expected mean square prediction error. The optimal model reduction turns out, under a certain assumption, to correspond to the statistical model for partial least squares discussed by the author elsewhere, and under a certain specific condition, the partial least squares predictors is proved to be good compared to all other predictors. It is also proved in this article that the situation with two different model reductions can be fit into a quantum mechanical setting.  Thus, the article contains a synthesis of three cultures: mathematical statistics as a basis, algorithms introduced by chemometricians and used very much by applied scientists as a background, and finally, notions from quantum foundation as an alternative point of view. 
\end{abstract}

\underline{Keywords:}
linear prediction;
model reduction;
optimal model;
optimal prediction;
partial least squares regression;
PLS-like methods;
quantum theory.

\section{Introduction}

There exists a large number of different statistical methods for the linear prediction of a single variable $y$ from $p$ variables $x_1,...,x_p$. The user of statistics is often left to choose the method that is familiar to him or her, or the method for which he/she has access to the relevant software. When $p$ is small, multiple linear regression is the method to choose, but in many practical applications, $p$ is large, often larger than the number $n$ of units where data is available.

For this situation, partial least squares (PLS) regression is a method that is emerging and recommended also by some  statisticians. The method was developed by chemometricians. It is related to Herman Wold's theory of latent variables, but in a regression setting, it is a simple, well-defined algorithmic method; see Appendix 1, or for a standard reference, Martens and N\ae s (1989). PLS regression has grown popular among very many applied researchers; see the review by Mehmod and Ahmed (2015). It was linked to a statistical model by Helland (1990), see also Helland (1992, 2001), N\ae s and Helland (1993), and Helland and Alm\o y (1994). The model was generalized to the case of multivariate $y$ and tied to Dennis Cook's envelope model in Cook et al. (2013). More recently, Cook and Forzani (2019) have studied the asymptotics of PLS regression as $n$ and $p$ tend to infinity and have given strong evidence that this should be the method of choice in the case of abundant regression, where many of the predictors $x_i$ contribute information about the response $y$.

For the present article, the motivation for PLS regression and related methods that are advocated in Helland et al. (2012) is particularly relevant. Here the random $x$ model is the point of departure, and a reduced model is approached through the group $G$ of rotations of the eigenvectors for the $x$-covariance matrix together with scale transformations for the regression coefficients. Below, this group and the orbits of the group will play a fundamental role. (An orbit of $G$ acting on some space $\Omega$ is the set $\{g\theta \}$ as $g$ runs through $G$ for some fixed $\theta\in\Omega$. It is easy to see that, if $\theta_1$ belongs to this orbit, then $\{g\theta_1 \}=\{g\theta \}$.)

A completely different area where multiple linear regression may be one of the building blocks, is machine learning, an important part of artificial intelligence. There is a large literature on machine learning and also a growing literature on the connections between artificial intelligence and statistical modeling. Of special interest for the present article is that there recently have been several investigations related to links between machine learning and quantum-mechanical models, see the review article by Dunjko and Briegel (2019). Given these investigations and the strong link between machine learning and statistics, it is strange that there has up to now been very little published research on possible links between statistical modeling and quantum mechanics. It is one purpose of the present article to discuss such a link. See also parts of the book by Helland (2021).

After that book was finished, this author has written several articles on the foundation of quantum theory, and some of these have been published in leading physics journals. My final approach towards the foundation is now given in Helland (2024a, b, c, d). This is also part of the basis for the present article.

The plan of the article is as follows: In Section 2, I summarize parts of the theory proposed in Helland (2024a), giving an alternative foundation of quantum theory. In the rest of the article, this is used in a statistical setting. In Section 3, the setting is specified to linear prediction, and a specific model reduction $\theta$ is defined in relation to a chosen dimension $m$, the model reduction which, according to my earlier articles on this topic, gives the statistical model corresponding to partial least squares (PLS) regression. This model is elaborated on in Section 4, where it is shown to correspond to a concrete quantum-mechanical setting, giving an operator $A^\theta$ corresponding to $\theta$, an operator defined on a Hilbert space $\mathcal{H}$. 

A main aim of this paper is to show that the PLS model is optimal under some assumptions, and for this purpose, a general model reduction $\eta$ with dimension $m$ is defined in Subsection 4.2. Then, in Section 5, the first optimality theorem, Theorem 3, is stated, interpreted, and proved. In Section 6, criteria for optimality of the PLS model are discussed. In Section 7, a possible basis for discussing the optimality of PLS-type regression compared to other methods like ridge regression is discussed. The discussion is completed by giving a concrete criterion for optimality of ordinary PLS regression in Section 8, and in Section 9, some concluding remarks are given.

Readers who are only interested in the optimality properties of partial least squares regression may concentrate on the theorems of the last few sections. The main conclusions of this article are given in Theorem 5 and Theorem 7.

\section{On quantum foundation.}

The present Section will contain some technicalities, which may be skipped in the first reading.

Traditional quantum mechanics is very formal in the sense that it is based on a underlying complex Hilbert space where physical variables are associated with self-adjoint operators and where states of a system are described by normalized vectors in this Hilbert space. For introductions to this theory, see textbooks like Ballentine (1989) or Sakurai (1994).

Recently, several physicists have tried to find a less formal foundation behind this formalism, see for instance Hardy (2015) or Masines and M\"{u}ller, (2011). Parts of my own approach will be sketched below.

\subsection{The basic theorem}

The fundamental notion in my approach towards quantum foundations is that of a \emph{theoretical variable} connected in a given context to an observer or to a communicating group of observers. In Helland (2024a, b, c, d) these variables were mostly physical variables. Later in the present article, they will be statistical parameters relative to some model. The essence of the theory turns out to be the same.

Theoretical variables are divided into \emph{accessible} and \emph{inaccessible} variables. In Helland (2024a) a physical variable was  said to be accessible if the observer(s) in principle in some future can obtain as accurate values of this variable as he/she/they wish to. In this article, I will let the variables be statistical parameters, let the inaccessible variables be parameters that are too extensive to be estimable with the available data, and let the accessible variables be parameters that can be estimated. Again, the theory from Helland (2024a) can be adapted. From a mathematical point of view, I require that if $\lambda$ is an accessible variable and $\theta=f(\lambda)$ for some function $f$, then $\theta$ is also accessible.

One postulate from Helland (2024a) is crucial for the theory there. I assume that there is a big inaccessible variable $\phi$, varying in a space $\Omega_\phi$, and I assume that all accessible variables can be seen as functions of this $\phi$.  In the present article, the inaccessible parameter $\phi$ will be given a very concrete definition. The article considers situations where there is a total parameter which cannot be estimated, and $\phi$ is this total parameter.

Given this postulate in Helland (2024a), the theory is completely rigorous from a mathematical perspective. First, consider the following definitions:
\bigskip

\begin{definition}\label{de1}
\textit{The accessible variable $\theta$ is called \emph{maximal} if the following holds: If $\theta$ can be written as $\theta=f(\psi)$ for a function $f$ that is not surjective, the theoretical variable $\psi$ is not accessible. In other words: $\theta$ is maximal under the partial ordering defined by $\alpha\le \beta$ iff $\alpha=f(\beta )$ for some function $f$.}\end{definition}

\begin{definition}\label{de2} \textit{Let $\theta$ and $\eta$ be two maximal accessible variables in some context, and let $\theta=f(\phi)$ for some function $f$. If there is a transformation $k$ of $\Omega_\phi$ such that $\eta(\phi)=f(k\phi)$, we say that $\theta$ and $\eta$ are \emph{related} (relative to this $\phi$). If no such $k$ can be found, we say that $\theta$ and $\eta$ are non-related relative to the variable $\phi$.}\end{definition}
\bigskip

The following theorem is then derived:

\begin{theorem}\label{th1}
\textit{Consider a context where there are two different related maximal accessible variables $\theta$ and $\eta$. Assume that both $\theta$ and $\eta$ are real-valued or real vectors, taking at least two values.  Make the following additional assumption:}
\smallskip

\textit{On one of these variables, $\theta$, there can be defined a transitive group of actions $G$ with a trivial isotropy group and with a left-invariant measure $\nu$ on the space $\Omega_\theta$.}
\smallskip

\textit{Then there exists a Hilbert space $\mathcal{H}$ connected to the situation, and to every (real-valued or vector-valued) accessible variable there can be associated a symmetric operator on $\mathcal{H}$.}
\end{theorem}

A group $G$ is called transitive on the space $\Omega_\theta$ if for every $\theta_1, \theta_2\in\Omega_\theta$, there exists a $g\in G$ such that $\theta_2 = g\theta_1$; that is, the group has only one orbit. The isotropy group connected to $\theta\in\Omega_\theta$ is the set of $g\in G$ such that $g\theta=\theta$. For a transitive group, the isotropy groups connected to different $\theta$ are in one-to-one correspondence. The measure $\mu$ on the measurable space $(\Omega_\theta ,\mathcal{F})$ is left-invariant if $\mu (gB)=\mu(B)$ for every $B\in\mathcal{F}$ and $g\in G$.

An operator $A$ is symmetric if $(A\bm{v},\bm{v})=(\bm{v}, A\bm{v})$ for every vector $\bm{v}$ in the domain of $A$. This is closely related to the more technical properties of being self-adjoint; see Hall (2013) for a very thorough discussion.

Here, $(\bm{u},\bm{v})$ is the scalar product in the Hilbert space. In the finite-dimensional case it can be written as $(\bm{u},\bm{v})=\bm{u}^\dagger \bm{v}$, where the vectors are column vectors, and $\bm{u}^\dagger$ is the complex conjugate row vector corresponding to $\bm{u}$. This scalar product is very important in quantum mechanics.

In Helland (2024a) an additional assumption regarding a certain unitary representation of the group $G$ was made. In Helland (2024d) it was shown that this assumption is not needed. 

Theorem 1 shows that an essential part of the Hilbert space formalism of quantum theory follows from weak assumptions. Further theorems and implications are given in the above references. I mention only here that in the discrete case, eigenvalues of the operators of Theorem 1 are the possible values of the associated theoretical variables and the quantum states can be given by the eigenvectors of the operators. These eigenvectors are in one-to-one correspondence with focused questions of the form `What is $\theta$?' together with sharp answers of the form `$\theta=u$'. Continuous variables may be treated in a similar way if we allow `eigenvectors' in the form of delta-functions.

Finally, probabilities are introduced by Born's rule. In the simplest case a probability is given by the square of a probability amplitude - a scalar product between eigenvectors of two different operators.

The important assumption in Theorem 1 is that in the given setting there exist two different maximal accessible variables, what Niels Bohr called complementary variables. An example from physics is the position and the momentum of a particle. By Heisenberg's inequality, these two accessible variables are maximal. One purpose of the present article is to give an example of a statistical application of the same theory. For brevity, the theory for this application is not developed much further here than verifying that the assumptions of Theorem 1 are satisfied.

As already mentioned, any relevant statistical context will then be one where the total parameter $\phi$ is so extensive that it cannot be estimated from the available data. We can then consider two different maximal accessible parameters $\theta =\theta(\phi)$ and $\eta=\eta(\phi)$.

\section{A setting for linear prediction.}

Consider a statistical setting with a large number $p$ of possible predictor variables $\bm{x} =(x_1,...,x_p)'$ and a response $y$. Assume that these variables have a joint distribution, and that we have observed $n$ samples from this distribution. For simplicity, I will assume in this article that all variables are centered on zero expectation, See Appendix 1 for centering.

This introduces the following parameters: $\mathbf\mathit{\Sigma}_{xx}=\mathrm{cov}(\bm{x})$, $\bm{\sigma}_{xy} = \mathrm{cov}(\bm{x},y)$, $\sigma^2 =\mathrm{var}(y|\bm{x})=\mathrm{var}(y)-\bm{\sigma}_{xy}'\mathbf\mathit{\Sigma}_{xx}^{-1}\bm{\sigma}_{xy}$ and $\bm{\beta}=(\beta_1,...,\beta_p)'=\mathbf\mathit{\Sigma}_{xx}^{-1}\bm{\sigma}_{xy}$. Let the collection of these parameters be denoted by $\phi$.

Let us assume that we know a new vector $\bm{x}_{new}$ with the same distribution as $\bm{x}$, and want to predict the $y$ corresponding to this vector. By well-known statistical theory, see Hastie et al. (2009), the best linear predictor, if $\bm{\beta}$ is known, is given by $\widehat{y}=\bm{\beta}\cdot\bm{x}_{new}$.

Throughout this article, consider a statistician $B$. He has data $\bm{X},\bm{y}$, consisting of $n$ samples from the above distribution, and wants to estimate $\bm{\beta}$. Since $p$ is large and $n$ may be moderate, the above set of parameters may be too large for him. He may consider two estimators $\widehat{\theta}$ and $\widehat{\eta}$, both based upon parameter reduction. 

Specifically, the estimator $\widehat{\theta}$ is based on the following model reduction.

Let $\bm{d}_1,...,\bm{d}_p$ be the normalized eigenvectors of $\mathbf\mathit{\Sigma}_{xx}$, assumed to be positive definite, and consider the decomposition
\begin{equation}
\bm{\beta}=\sum_{j=1}^p \gamma_j \bm{d}_j.
\label{11}
\end{equation}

In agreement with the PLS model in Helland et al. (2012) and the envelope model of Cook et al. (2013), fix a number $m$, and consider estimation/prediction under the hypothesis:
\smallskip

$H_m$: \emph{There are exactly $m$ nonzero terms in (\ref{11}).}
\smallskip

There are two mechanisms by which this number of terms can be reduced: 1) Some $\gamma$'s are zero at the outset. 2) There are coinciding eigenvalues of $\mathbf\mathit{\Sigma}_{xx}$, and then the eigenvectors may be rotated in such a way that there is only one in the relevant eigenspace that is along $\bm{\beta}$.

Considering $H_m$ as a model reduction, it is shown in Helland (1990) and Cook et al. (2013) that it can be formulated in the following equivalent way: Let $\theta=\theta_m$ be defined by the Krylov set $\bm{\sigma}_{xy}, \mathbf\mathit{\Sigma}_{xx}\bm{\sigma}_{xy}. \mathbf\mathit{\Sigma}_{xx}^{2}\bm{\sigma}_{xy},..., \mathbf\mathit{\Sigma}_{xx}^{m-1}\bm{\sigma}_{xy}$, then $m$ is the smallest number such that $\bm{\beta}$ is a linear function of $\theta_m$. 

For the purpose of this article, however, we will define $\theta=(\gamma_1 \bm{d}_1,...,\gamma_m\bm{d}_m)$, with all $\gamma_i\ne 0$, and define the model under $H_m$:
\begin{equation}
\beta=\beta_m =\beta(\theta)=\sum_{j=1}^m \gamma_j \bm{d}_j.
\label{12}
\end{equation}

Note that (\ref{11}) is invariant under permutations of the terms, so we might as well take the non-trivial terms to be the first $m$ terms.

The model reduced by the hypothesis $H_m$ is equivalent to the PLS model of Helland (1990), and is a special case of the envelope model of Cook (2018). Note that an equivalent formulation of the PLS model is that the population PLS algorithm stops automatically at step $m$; see Appendix 1.

It is interesting that this model reduction may be connected to a particular group $K$ acting on the parameter $\bm{\beta}$, also involving $\mathbf\mathit{\Sigma}_{xx}$, that is, a group on the parameter space $\Omega_\phi$:

\begin{definition}\label{de4} \textit{Let the group $K$ be defined by orthogonal matrices acting on all the vectors $\bm{d}_j$ in (\ref{11}), and in addition separate scale transformations of the parameters $\gamma_j$: $\gamma_j\mapsto g_j(\gamma_j)$ for some bijective continuous functions $g_j$ on the line.}\end{definition}

The first part is equivalent to orthogonal transformations of $\mathbf\mathit{\Sigma}_{xx}$. It can be induced by rotating the vector $\bm{x}$, and in addition by changing the sign of this vector.

\begin{theorem}\label{th5} \textit{If and only if  the bijective continuous functions $g_j$ are such that $g_j(0)=0$, the orbits of the group $K$ are determined by: a given $m$ and the hypothesis $H_m$.}
\end{theorem}

\begin{proof}
If and only if $g(0)=0$, the group on $\gamma$ defined by $\gamma\mapsto g(\gamma)$ has two orbits: 1) the single value $\gamma=0$; 2) the set of all $\gamma$ such that $\gamma\ne 0$. Going to the whole group $K$, this implies an orbit where $p-m$ of the $\gamma_k$'s are zero and $m$ of the $\gamma_k$'s are non-zero. That is, exactly the hypothesis $H_m$.
\end{proof}

\begin{definition}\label{de1} \textit{Define the group $G$ acting on $\theta$ by orthogonal transformations of the vectors $\bm{d}_j$ in (\ref{12}) and in addition separate linear scale transformations of the parameters $\gamma_j$: $\gamma_j\mapsto \alpha_j\gamma_j$ with $\alpha_j > 0$.}\end{definition}

Taking into account that the changes of sign $\gamma_j\mapsto -\gamma_j$ may also be obtained by orthogonal transformations of the $\bm{d}_i$'s, this implies that the group $G$ is transitive, and it also has a trivial isotropy group. The elements $g\in G$ are then in one-to-one correspondence with the values of $\theta$.

\section{A quantum-mechanical setting related to a model reduction}

\subsection{The group}

 Let $\theta$ in any case be a function of the nonzero parameters $\gamma_1,...,\gamma_m$ and the $\mathbf\mathit{\Sigma}_{xx}$-eigenvectors $\bm{d}_1,..., \bm{d}_m$, all normalized: $\theta = (\gamma_1\bm{d}_1,...,\gamma_m\bm{d}_m)$. The elements of the group $G$ are given by 1) a matrix $\bm{O}$ with orthonormal columns such that $(\bm{d}_1,...\bm{d}_m)\mapsto \bm{O}(\bm{d}_1,...,\bm{d}_m)$; 2) positive scalars $\alpha_j$ giving scale transformations $\gamma_j\mapsto\alpha_j\gamma_j$.

A left-invariant measure of the scale transformation $\gamma\mapsto\alpha\gamma$ is given by $\mu(d\gamma)=d\gamma/\gamma$ on $\{\gamma:\gamma>0\}$. Negative signs of $\gamma$ may be tackled through a sign change of $\bm{d}$, so this implies that $\mu$ can be extended to the whole line except $\gamma=0$. The left-invariant measure on the $m$-dimensional rotation group is given by the uniform measure $\sigma$ on the $m$-dimensional sphere in $\mathbb{R}^p$, and the change of sign by $\nu(+)=\nu(-)=1/2$. This determines the measure $\nu$ on $\Omega_\theta$.

Theorem 1 gives, in general, a Hilbert space and, in particular, an operator $A^\theta$ on this Hilbert space under certain conditions for the case when we, in addition, have a complementary parameter $\eta$.
 One of the conditions behind the theorem is the existence of a transitive group $G$ acting upon $\theta$. This can be taken as the group defined above.

\subsection{Another model reduction}

Again, consider a statistical setting with a large number $p$ of possible predictor variables $\bm{x} =(x_1,...,x_p)'$ and a response $y$. Assume that these variables have a joint distribution and that we have observed $n$ samples from this distribution.

Again, this introduces the parameters: $\mathbf\mathit{\Sigma}_{xx}=\mathrm{cov}(\bm{x})$, $\bm{\sigma}_{xy} = \mathrm{cov}(\bm{x},y)$, $\sigma^2 =\mathrm{var}(y|\bm{x})=\mathrm{var}(y)-\bm{\sigma}_{xy}'\mathbf\mathit{\Sigma}_{xx}^{-1}\bm{\sigma}_{xy}$ and $\bm{\beta}=(\beta_1,...,\beta_p)'=\mathbf\mathit{\Sigma}_{xx}^{-1}\bm{\sigma}_{xy}$. Let again the collection of these parameters be denoted by $\phi$, varying in some space $\Omega_\phi$.

There are many ways to perform a model reduction in a prediction context. Assume that the statistician $B$ also considers another reduction $\eta$ based upon the same inaccessible parameter $\phi$, so $\eta=\eta(\phi)$.

More specifically, I assume: Fix some number $m$, for $j=1,...,m$ let $\eta_j(\cdot)$ be a $p$-dimensional vector function defined on $\Omega_\phi$, and put $\eta(\phi)=(\eta_1(\phi),...,\eta_m(\phi))$. For linear prediction, let the reduced regression parameter be $\beta'_m =\beta(\eta,\phi')$ for some function $\beta(\cdot)$, where $\phi'$ is chosen so that $(\eta(\phi),\phi')$ is in one-to-one correspondence with $\phi$. I will suppose that $\beta$ can be estimated under the hypothesis
\bigskip

\textit{$H_m':$ $\beta'_m =\beta(\eta,\phi')$ is accessible, but maximally so: If $\eta=f(\xi)$ for some function $f$ which is not surjective, then  $\beta(\xi,\phi')$  is not accessible.}
\bigskip

This should be compared to the hypothesis $H_m$ that was made in connection to the specific reduction $\theta$ of Section 4. Note that I assume that $\eta$ also, in relation to the regression coefficient $\beta$, has just $m$ vector components, and that both $\theta$ and $\eta$ can be seen as maximal accessible parameters for $A$. Let $M$ be a fixed group acting on $\Omega_\phi$ which transforms such sets of $m$ $p$-dimensional vectors into other sets of $m$ $p$-dimensional vectors.

Then, assuming some fixed value $\theta_1$ of $\theta$, we can first find a $\phi_1 \in \Omega_\phi$ such that $\theta(\phi_1)=\theta_1$. Given some fixed value $\eta_2$ of $\eta$ and a $\phi_2\in \Omega_\phi$ such that $\eta(\phi_2)=\eta_2$, then either $\phi_1$ and $\phi_2$ lie on the same orbit of $M$, or they belong to different orbits. In the first case, there is a $k\in M$ such that $\phi_2 = k\phi_1$. In the second case, there is an element $k\in M$, a $\phi_3\in\Omega_\phi$ and a bijective function $f$ such that $\phi_2 =f(\phi_3)$ and $\phi_3 = k\phi_1$. Since bijective functions in $\Omega_\phi$ imply equivalent model reductions $\eta(\phi)$, this means that one can without loss of generality assume a transformation $k$ such that $\eta_2=\eta(\phi_2)=\theta(k\phi_1)$ while $\theta_1 =\theta(\phi_1)$. Since $\theta_1$ and $\eta_2$ were arbitrarily chosen, this implies that $\theta$ and $\eta$ are related as defined in Definition 2. The crucial assumptions are that both parameters have the same dimension and are defined as functions on the same space $\Omega_\phi$.

But this implies that the conditions of Theorem 1 are satisfied. There exist operators $A^\theta$ and $A^\eta$ in the same Hilbert space $\mathcal{H}$ corresponding to the two model reductions. 

Later, we will study the two parametric functions $\beta(\theta)$ and $\beta(\eta)$ (strictly speaking $\beta(\eta,\phi')$) in relation to the true regression parameters $\beta$.

\section{An optimality theorem for model reduction}

Let us assume that $\theta=\theta(\phi)$ is the PLS model reduction with a fixed number $m$ of relevant components as described in Section 3, and let $\eta=\eta(\phi)$ be another $m$-dimensional model reduction as described in Subsection 4.2. Here, $\phi$ is the parameter of the full model. Assume that there is a continuous group $M$ acting upon $\Omega_\phi$ which transforms the specific sets of $m$ $p$-dimensional vectors into similar sets of $p$-dimensional vectors.

The purpose of this Section is to investigate when the PLS model gives the best model reduction for prediction. Seen from an asymptotic point of view, there are  several criteria in the literature for when PLS regression performs well in a prediction setting. In Helland and Alm\o y (1994) a criterion was formulated in terms of relevant eigenvalues. Cook and Forzani (2019) indicated that PLS performs well in abundant regression where many predictors contribute information about the response. In this article, I will make exact computations and formulate a relatively concrete criterion. For brevity, write $\beta(\eta)=\beta(\eta,\phi')$, where $(\eta,\phi')$ is in one-to-one correspondence with $\phi$.
\bigskip

\textbf{Assumption $A$.} \textit{Let $\theta$ with $m$ components be the PLS model assumption, let $\eta$ denote another $m$-dimensional model reduction, and let $\bm{\beta}$ be the true regression coefficient. Assume that, relative to the distribution of $\bm{x}$}
\begin{equation}
\mathrm{Cov}([(\bm{\beta}(\eta)-\bm{\beta}(\theta))\cdot\bm{x}], [(\bm{\beta}(\eta)+\bm{\beta}(\theta)-2\bm{\beta})\cdot\bm{x}])>0
\label{32a}
\end{equation}
\smallskip

Note that if $\bm{\beta}(\theta)$ is close to the true regression coefficient $\bm{\beta}$, this is guaranteed to hold; see later.

I will prove the following theorem:

\begin{theorem}\label{th6} \textit{Let $(\bm{x},y)$ have a joint distribution with all second order parameters given by the parameter $\phi$. Assume that all variables have expectation $0$ and that the $x$-covariance matrix $\mathbf\mathit{\Sigma}_{xx}$ is positive definite. Make the Assumption $A$. Then the $m$-dimensional reduction of $\phi$ given by the PLS-model is better than the $m$-dimensional reduction given by $\eta$, in the sense that $\mathrm{E}_{(\bm{x},y)}(y-\beta(\cdot)\cdot \bm{x})^2$ is minimized. Conversely, if the PLS model gives a better prediction than $\eta$, then Assumption $A$ must hold.}\end{theorem}

\begin{proof}
As a point of departure, let $\eta$ be any $m$-dimensional model reduction satisfying Assumption $A$. As shown in Subsection 5.2, we can write $\eta(\phi)=\theta(k\phi)$ for some $k\in M$, and since $\theta$ is a continuous function of $\phi$, it is meaningful to let $k$ approach the identity, that is, let $\eta(\phi)\rightarrow\theta(\phi)$. 
 
Since the eigenvectors $\bm{d}_j$ of $\Sigma_{xx}$ form a basis for $\mathbb{R}^p$, we have 
\begin{equation}
\beta_{m}(\eta)=\sum_{j=1}^p \delta_j(\eta)\bm{d}_j 
\label{30}
\end{equation}

The $\delta_j$'s are functions of $\eta$, and may be seen as close to some $\gamma_j$'s when $\eta$ is close to $\theta$. Note that the terms in (\ref{30}) can be permuted, so without loss of generality, we can let the first $m$ terms correspond to the PLS solution (\ref{12}). If the hypothesis $H_m$ holds, the $\gamma_j$'s for $j=m+1,...,p$ are zero. 

Let $\beta(\eta)=\beta(\theta)+\bm{e}(\phi)$. Define $\tau(\eta(\phi))= \mathrm{E}(y-\beta(\eta)\cdot\bm{x})^2$. Then
\begin{equation}
\tau(\eta(\phi)) = \mathrm{E}(y-\beta(\theta)\cdot\bm{x})^2- 2\mathrm{E}(y-\beta(\theta)\cdot\bm{x})(\bm{e}\cdot\bm{x})+ \bm{e}'\Sigma_{xx}\bm{e}.
\label{23}
\end{equation}

The cross-term here may be written
\begin{equation}
\bm{\sigma}_{xy}'\bm{e}-\beta(\theta)'\Sigma_{xx}\bm{e}=(\beta-\beta(\theta))'\Sigma_{xx}(\beta(\eta)-\beta(\theta)),
\label{24}
\end{equation}
So
 \begin{equation}
 \tau(\eta(\phi))=E_{(\bm{x},y)}(y-\beta(\theta)\cdot\bm{x})^2 + F(\phi)=\tau(\theta(\phi))+F(\phi),
 \label{25}
 \end{equation}
 where $\beta(\theta)$ is given by (\ref{12}) and
 \begin{equation}
 F(\phi)=(\beta(\eta)+\beta(\theta)-2\beta)'\Sigma_{xx}(\beta(\eta)-\beta(\theta)),
 \label{26}
 \end{equation}
where $\beta$ is the true regression vector. Comparing this with (\ref{32a}) concludes the proof of Theorem 6. Since all calculations are exact, there is an if and only if here.\end{proof}
 \smallskip
 
 \textbf{Corollary 1.} \textit{Under the hypothesis $H_m$ of Section 4, the PLS regression model always gives a best model reduction for linear prediction.}
\smallskip

\begin{proof}
Under $H_m$ we have $\beta=\beta(\theta)$, and (\ref{26}) is non-negative for all $\eta$.
\end{proof}
\smallskip

\textbf{Corollary 2.} \textit{Assume that $\mathrm{Var}((\beta-\beta(\theta)) \cdot \bm{x})< \frac{1}{4}\mathrm{Var}((\beta(\eta)-\beta(\theta)) \cdot \bm{x})$. Then the PLS model will give better linear predictions than the model reduction $\eta$. }

\begin{proof}
(\ref{26}) can be written
\begin{equation}
F(\phi) =(\beta(\eta)-\beta(\theta))'\Sigma_{xx}(\beta(\eta)-\beta(\theta))-2(\beta-\beta(\theta))'\Sigma_{xx}(\beta(\eta)-\beta(\theta)).
\label{27}
\end{equation}
By a version of the Caucy-Schwarz inequality, this is guaranteed to be positive if $(\beta-\beta(\theta))'\Sigma_{xx}(\beta-\beta(\theta))<\frac{1}{4}(\beta(\eta)-\beta(\theta))'\Sigma_{xx}(\beta(\eta)-\beta(\theta))$.
\end{proof}

\section{On optimality of the PLS model under any model reduction}

Let us consider the situation with two different model reductions $\theta$ and $\eta$, both corresponding to reductions to dimension $m$, as specified with the hypothesis $H_m$ of Section 3 and the hypothesis $H'_m$ of Subsection 4.2. We are interested in conditions under which the PLS model always is best in terms of mean square prediction error. 

\begin{theorem}
Assume that
\begin{equation}
4\mathrm{E}_\theta (\beta-\beta(\theta))'\Sigma_{xx}(\beta -\beta(\theta)) <\mathrm{E}_\theta (\beta(\eta)-\beta(\theta))'\Sigma_{xx}(\beta(\eta)-\beta(\theta)).
\label{37}
\end{equation}

Then

\begin{equation}
\mathrm{E}_\theta\mathrm{E}_{(x,y)}(y-\beta(\theta)\cdot \bm{x})^2 < \mathrm{E}_\theta\mathrm{E}_{(x,y)}(y-\beta(\eta)\cdot \bm{x})^2.
\label{38}
\end{equation}
\end{theorem}
\smallskip

\begin{proof}
Repeat the proof of Theorem 3 and of Corollary 2 of Section 5 with the expectation over $\theta$ taken in all equations. The necessary generalization of the Cauchy-Schwarz inequality follows by the same proof as for the ordinary inequality.
\end{proof}
\smallskip

We want to study the criterion (\ref{37}) more closely. Since $\beta(\theta)=\sum_{j=1}^m \gamma_j \bm{d}_j$, the left-hand side is just
\begin{equation}
4\mathrm{E}_\theta\sum_{j=m+1}^p \gamma_j^2\lambda_j = 4\sum_{j=m+1}^p \gamma_j^2\lambda_j,
\label{39}
\end{equation}
assuming that the $\gamma_j$'s are independent, where here $\{\lambda_j\}$ are the irrelevant eigenvalues of $\Sigma_{xx}$, those not affected by the model reduction $\theta$.

We will assume that $\Sigma_{xx}$ is positive definite, so that $\beta$, and any model reduction of $\beta$, is spanned by the uniquely defined eigenvectors $\{\bm{d}_j\}$ of $\Sigma_{xx}$. In particular

\begin{equation}
\beta(\eta)=\sum_{j=1}^p \zeta_j \bm{d}_j .
\label{41}
\end{equation}

The right-hand side of the inequality (\ref{37}) is then bounded below by
\begin{equation}
\sum_{j=1}^m \mathrm{E}_\theta(\zeta_j-\gamma_j)^2 \lambda_j .
\label{40}
\end{equation}

Our aim is to find a criterion under which the PLS model reduction is better in some sense than any other model reduction. This means that the parameters $\zeta_j$ in (\ref{40}) are completely arbitrary.

Note that each $\gamma_j$ can be seen as a function of $\theta$: $\gamma_j = \sqrt{|\gamma_j\bm{d}_j |^2}\mathrm{sign}(\gamma_j)$, where the sign is determined as follows: Since $\gamma_j \bm{d}_j =(-\gamma_j)(-\bm{d}_j)$, each pair $(\gamma_j ,\bm{d})$ is counted twice in $\theta$. We can let one of these repetitions correspond to a positive $\gamma_j$, the other to a negative $\gamma_j$.

Let now the basic parameter $\theta$ have some probability distribution, which implies a probability distribution of $\gamma_1,...,\gamma_m$. Then the criterion (\ref{37}) is satisfied  over $\theta$ for a model reduction $\eta$ if
\begin{equation}
\mathrm{E}_\theta \sum_{j=1}^m (\zeta_j-\gamma_j)^2 \lambda_j> 4 \sum_{j=m+1}^p \gamma_j^2\lambda_j.
\label{42}
\end{equation}

For each $j$ we have that $\mathrm{E}_{\gamma_j}(\zeta_j-\gamma_j)^2 \ge  \mathrm{E}_{\gamma_j}(\mu_j-\gamma_j)^2$, where $\mu_j=\mathrm{E}_{\gamma_j} (\gamma_j)$. So, taking a lower bound on the left-hand side of (\ref{42}), we see that the criterion (\ref{37}) is satisfied for every possible reduction $\eta$ if
\begin{equation}
\mathrm{E}_\theta \sum_{j=1}^m (\gamma_j-\mu_j)^2 \lambda_j > 4 \sum_{j=m+1}^p \gamma_j^2\lambda_j.
\label{43}
\end{equation}

\begin{theorem}
\textit{Assume that}
\begin{equation}
\sum_{r=1}^m \lambda_j \mathrm{E}_{\gamma_j} (\gamma_j-\mu_j)^2 > 4\sum_{j=m+1}^p \gamma_j^2\lambda_j .
\label{46}
\end{equation}
\textit{Then the criterion (\ref{37}) for optimality of the PLS model is satisfied, and .}

\begin{equation}
\mathrm{E}_\theta\mathrm{E}_{(x,y)}(y-\beta(\theta)\cdot \bm{x})^2 < \mathrm{E}_\theta\mathrm{E}_{(x,y)}(y-\beta(\eta)\cdot \bm{x})^2.
\label{46a}
\end{equation}
\end{theorem}

This indicates strongly that $H_m$ gives a good model reduction when the relevant eigenvalues of $\Sigma_{xx}$ are substancially larger than the irrelevant ones. It is also significant that the variances of the relevant regression coefficients are fairly large compared to the squares of the irrelevant regression coefficients.

Note that the criterion (\ref{46}) is only connected to the PLS model reduction. If this criterion is satisfied for some $m$, the PLS model reduction is better than \emph{all} other model reductions. Also, note that the lefthand side of (\ref{46}) is increasing with $m$, and the righthand side is decreasing. Thus it seems reasonable that the criterion in most situations is satisfied if $m$ is large enough.

In general, the probability distributions of the $\gamma_j$'s will depend on the situation. As a first tentative situation, let us first assume a probability distribution of $\gamma_j$ which is close to the left-invariant measure $\mu(d\gamma)=d\gamma/\gamma$ under the group $G$. (See subsection 4.1.) This measure gives an improper distribution, and under a proper distribution close to this distribution, the lefthand side of (\ref{46}) can be made arbitrarily large. This indicates that under such circumstances, it will be easy to satisfy the criterion (\ref{37}) (for any $m$).

It will also be useful to consider the fact that this model reduction problem also, as indicated above, has links to quantum foundation. The relevant theory, given in Appendix 2, has the following interpretation: In a setting where the statistician $B$ thinks about two different model reductions $\theta$ and $\eta$, one can look at the expectations of the parametric functions $\xi_j (\theta)=(\gamma_j-\mu_j)^2$, given a non-informative prior on $\eta$, by first considering a finite-valued approximation of the parameters, and using the Born formula together with equation (\ref{18b}), and then taking the limit as the finite-valued approximate parameters approach the real parameters. This turns out to give the criterion (\ref{46}) again, but now with a rather concrete interpretation; see Appendix 2 for the argument. Thus this criterion, and hence the condition for optimality of the PLS model for large enough $m$, seems to be relevant under reasonable conditions.

\section{ A first condition for optimal linear prediction by PLS-like methods}

In the statistical literature, there are several methods proposed for linear prediction of a variable $y$ from many predictors, possibly related. One example is ridge regression with some given ridge parameter. Other examples are principal component regression and latent roots regression. 

In this section and the next one, I will investigate, in principle, when PLS-like methods are optimal in some sense in this large class of methods. In this section, I will fix a number $m$ and assume that the hypothesis $H_m$ (see Section 4) holds. With PLS-like methods, I will mean either ordinary PLS regression or related methods like Bayes PLS (see Helland et al., 2012) or maximal likelihood PLS (see Cook et al., 2013). The main assumptions are that the hypothesis $H_m$ is satisfied for some $m$, that an estimate of $\beta$ is done with $a\ge m$ components, and that there exists a probability distribution of the corresponding estimator, given $\theta$.

Assume that we want to find a good predictor of $y$ from a $p$-dimensional $\bm{x}$ based upon $n$ data $\bm{X},\bm{y}$. For simplicity, let all data variables be centered to zero expectation.

In the Theorem below, I consider a fixed PLS like method. The criterion used is expected mean square prediction error, where we take expectation over the variables in the data set, the future $\bm{x}$ and $y$ data and some distribution of the PLS parameter $\theta$.

\begin{theorem}\label{th10} \textit{Let $\widehat{\beta}$ be an arbitrary estimator of $\beta$. Then, for each  $a$ such that $m\le a<p$, letting    $\widehat{\beta_a}$ be constructed from a PLS-like method with $a$ components, assuming the hypothesis $H_m$, we have}

\textit{ If $\widehat{\beta}$ is sufficiently far from $\widehat{\beta_a}$, more concretely if}
\begin{equation}
\mathrm{E}_{\bm{X},\bm{y}}(\widehat{\beta}-\widehat{\beta_a})'\Sigma_{xx}(\widehat{\beta}-\widehat{\beta_a})>4\mathrm{E}_{\bm{X},\bm{y}}(\beta-\widehat{\beta_a})'\Sigma_{xx}(\beta-\widehat{\beta_a}),
\label{31}
\end{equation}
\textit{where $\beta$ is the true regression coefficient, then we have}
\begin{equation}
\mathrm{E}_{\bm{X},\bm{y}}\mathrm{E}_{\bm{x},y}(y-\widehat{\beta_a}\cdot\bm{x})^2 < \mathrm{E}_{\bm{X},\bm{y}}\mathrm{E}_{\bm{x},y}(y-\widehat{\beta}\cdot\bm{x})^2.
\label{32}
\end{equation}
\end{theorem}
\smallskip

\begin{proof}

Let $\mathrm{E}=\mathrm{E}_{\bm{X},\bm{y}}\mathrm{E}_{\bm{x},y}$. In analogy with the calculations of Section 5, we have
\begin{equation}
\mathrm{E}(y-\widehat{\beta}\cdot\bm{x})^2 = \mathrm{E}(y-\widehat{\beta_a}\cdot\bm{x})^2 +F,
\label{33}
\end{equation} 
where
\begin{equation}
F=\mathrm{E}(\widehat{\beta}+\widehat{\beta_a}-2\beta)'\Sigma_{xx}(\widehat{\beta}-\widehat{\beta_a})
\label{34}
\end{equation}
\begin{equation}
=\mathrm{E}(\widehat{\beta}-\widehat{\beta_a})'\Sigma_{xx}(\widehat{\beta}-\widehat{\beta_a})-2\mathrm{E}(\beta-\widehat{\beta_a})'\Sigma_{xx}(\widehat{\beta}-\widehat{\beta_a})
\label{35}
\end{equation}
\begin{equation}
 \ge \mathrm{E}(\widehat{\beta}-\widehat{\beta_a})'\Sigma_{xx}(\widehat{\beta}-\widehat{\beta_a})-2\sqrt{\mathrm{E}(\beta-\widehat{\beta_a})'\Sigma_{xx}(\beta-\widehat{\beta_a})\cdot \mathrm{E}(\widehat{\beta}-\widehat{\beta_a})'\Sigma_{xx}(\widehat{\beta}-\widehat{\beta_a})}
\label{36}
\end{equation}
by a variant of the Cauchy-Schwarz inequality.

Inspecting the inequality (\ref{36}), it follows that $F>0$, and hence (\ref{32}) holds if (\ref{31}) is satisfied. 
\end{proof}

The criterion (\ref{31}) will be further developed for the case of ordinary PLS regression in the next Section.

\section{On optimal linear prediction by ordinary PLS regression}

I will here complete the investigation of the situation where a PLS estimator $\widehat{\beta}_a$ with $a$ steps is compared with an arbitrary estimator $\widehat{\beta}$. The inequality (\ref{31}) has the same form as the inequality (\ref{37}), and part of the discussion from Section 9 can be carried over. 

The PLS regression vector with $a$ steps can be written as
\begin{equation}
\widehat{\beta_a}= \sum_{j=1}^a \widehat{\alpha_j}\widehat{\bm{e}_j},
\label{50}
\end{equation}
where $\{\widehat{\bm{e}_j}\}$ are given by the Krylov sequence: 
\begin{equation}
\widehat{\bm{e}_1}=\widehat{\bm{\sigma}_{xy}} ,\ \widehat{\bm{e}_2}=\widehat{\mathbf\mathit{\Sigma}_{xx}} \widehat{\bm{\sigma}_{xy}} ,\ \widehat{\bm{e}_3}=( \widehat{\mathbf\mathit{\Sigma}_{xx}} )^2\widehat{\bm{\bm{\sigma}}_{xy}},...,
\label{51}
\end{equation}
and $\{\widehat{\alpha_j}\}$ are suitable coefficients (see Appendix 1).

\begin{theorem}
 \textit{Let $\widehat{\bm{\beta}}$ be an arbitry estimator of $\bm{\beta}$, and assume that the hypothesis $H_m$ holds for some $m$. For some  $a$ such that $m\le a<p$, let    $\widehat{\bm{\beta}_a}$ be constructed by the PLS algorithm with $a$ components. Assume that }
 \begin{equation}
\mathrm{trace}(\mathbf\mathit{\Sigma}_{xx}\mathbf\mathit{W})\ge 4\mathrm{E}_{\bm{X},\bm{y}}(\bm{\beta}-\widehat{\bm{\beta}_a})'\Sigma_{xx}(\bm{\beta}-\widehat{\bm{\beta}_a}),
\label{47}
\end{equation}
where $\mathbf\mathit{W}=\sum_{j=1}^a\mathrm{E}_{\bm{X},\bm{y}} (\widehat{\alpha_j} \widehat{\bm{e}_j}-\bm{\mu}_j )(\widehat{\alpha_j} \widehat{\bm{e}_j}-\bm{\mu}_j )'$, and $\bm{\mu}_j = \mathrm{E}_{\bm{X},\bm{y}}(\widehat{\alpha_j} \widehat{\bm{e}_j})$.

\textit{Then we conclude that for any $\widehat{\bm{\bm{\beta}}}$ we have}
\begin{equation}
\mathrm{E}_{\bm{X},\bm{y}}\mathrm{E}_{\bm{x},y}(y-\widehat{\bm{\beta}_a}\cdot\bm{x})^2 \le \mathrm{E}_{\bm{X},\bm{y}}\mathrm{E}_{\bm{x},y}(y-\widehat{\bm{\beta}}\cdot\bm{x})^2.
\label{48}
\end{equation}
\end{theorem}

\begin{proof}
Start with (\ref{31}), repeated here as

\[\mathrm{E}_{\bm{X},\bm{y}}(\widehat{\bm{\beta}}-\widehat{\bm{\beta}_a})'\mathbf\mathit{\Sigma}_{xx}(\widehat{\bm{\beta}}-\widehat{\bm{\beta}_a})> 4\mathrm{E}_{\bm{X},\bm{y}}(\bm{\beta}-\widehat{\bm{\beta}_a})'\Sigma_{xx}(\bm{\beta}-\widehat{\bm{\beta}_a}).\]

By replacing $>$ by $\ge$ here, we can also allow $\widehat{\bm{\beta}}$ to be another PLS-like estimator.

The PLS regression vector with $a$ steps can be written as (\ref{50}). We can also write
\begin{equation}
\widehat{\bm{\beta}}=\sum_{j=1}^a \zeta_j \widehat{\bm{e}_j}+\bm{f},
\label{50a}
\end{equation}
where the resudual $\bm{f}$ is chosen in such a way that $ \widehat{\bm{e}_j}'\mathbf\mathit{\Sigma}_{xx}\bm{f}=0$ for $j=1,...,a$. More precisely, we project the vector $\widehat{\bm{\beta}}$ onto the space spanned by $\mathbf\mathit{D}=( \widehat{\bm{e}_1},..., \widehat{\bm{e}_a})$ by the skew projection  $ \mathbf\mathit{D} (\mathbf\mathit{D}''\mathbf\mathit{\Sigma}_{xx} \mathbf\mathit{D})^{-1} \mathbf\mathit{D}'\mathbf\mathit{\Sigma}_{xx} $, and let $\bm{f}$ be the residual from this.

This means that the cross terms in the expansion of 
\begin{equation}
(\widehat{\bm{\beta}_a}-\widehat{\bm{\beta}})'\mathbf\mathit{\Sigma}_{xx}(\widehat{\bm{\beta}_a}-\widehat{\bm{\beta}}) = (\sum_{i=1}^a(\widehat{\alpha_j}-\zeta_j) \widehat{\bm{e}_j}-\bm{f})'\mathbf\mathit{\Sigma}_{xx}(\sum_{i=1}^a(\widehat{\alpha_j}-\zeta_j) \widehat{\bm{e}_j}-\bm{f})
\label{51}
\end{equation}
vanish.

The left-hand side of (\ref{31}) is then bounded below by
\begin{equation}
\mathrm{E}_{\bm{X},\bm{y}}\sum_{j=1}^a (\widehat{\alpha_j}-\zeta_j)^2 \widehat{\bm{e}_j}'\mathbf\mathit{\Sigma}_{xx}\widehat{\bm{e}_j} = \sum_{j=1}^a \mathrm{E}_{\bm{X},\bm{y}} (\widehat{\alpha_j}\widehat{\bm{e}_j}-\zeta_j\widehat{\bm{e}_j})'\mathbf\mathit{\Sigma}_{xx} (\widehat{\alpha_j}\widehat{\bm{e}_j}-\zeta_j\widehat{\bm{e}_j}) .
\label{53}
\end{equation}

For each $j$, and whatever the coefficients $\zeta_j$ are, we have 
\[\mathrm{E}_{\bm{X},\bm{y}}(\widehat{\alpha_j}\widehat{\bm{e}_j}-\zeta_j\widehat{\bm{e}_j})'\mathbf\mathit{\Sigma}_{xx} (\widehat{\alpha_j}\widehat{\bm{e}_j}-\zeta_j\widehat{\bm{e}_j}) \ge \mathrm{E}_{\bm{X},\bm{y}}(\widehat{\alpha_j}\widehat{\bm{e}_j}-\bm{\mu_j})'\mathbf\mathit{\Sigma}_{xx} (\widehat{\alpha_j}\widehat{\bm{e}_j}-\bm{\mu}_j),\] 
where $\bm{\mu}_j = \mathrm{E}_{\bm{X},\bm{y}} (\widehat{\alpha_j}\widehat{\bm{e}_j}).$
So, (\ref{53}) is again bounded below by
\begin{equation}
 \mathrm{trace}[\Sigma_{xx} \sum_{j=1}^a\mathrm{E}_{\bm{X},\bm{y}} (\widehat{\alpha_j} \widehat{\bm{e}_j}-\bm{\mu}_j )(\widehat{\alpha_j} \widehat{\bm{e}_j}-\bm{\mu}_j )']
\label{54}
\end{equation}
.\end{proof}

\underline{Remarks}

The result could be seen in relation to the discussion of Helland and Alm\o y (1994), based upon simulations, which gives criteria when PLS regression is good compared to other estimators. It could also be compared to the statements of Cook and Forzani (2019) concerning when PLS should be the choice to perform.

We should remark the following: Note that the criterion (\ref{47}) only depends upon the PLS estimator. If it is satisfied, this estimator is good compared to all other estimators. The right-hand side of (\ref{47}) is small when the estimator $\widehat{\bm{\beta}}_a$ is close to the parameter $\bm{\beta}$. But then $\mathbf\mathit{W}$, which can be seen as the diagonal contribution to the matrix  $V(\widehat{\bm{\beta}_a})$, will also be small. The validity of (\ref{47}) for large $p$ and $n$ could in principle be studied by refering to asymptotic expansions like the ones in Cook and Forzani (2019), but this will not be pursued here.

The same criterion can also be studied 1) by the quantum model; 2) by computer simulations. Here, the package Simrel, see S\ae b\o \ et al. (2015), seems to be very suitable.

Thus, the conclusion seems to be: The PLS estimator $\widehat{\bm{\beta}}_a$ is better in terms of mean square prediction than an arbitrary estimator $\widehat{\bm{\beta}}$ if the following two conditions are satisfied: 1) The hypothesis $H_m$ holds for some $m$. A sufficient condition for this is given in Theorem 5 of Section 6. 2) The condition (\ref{47}) holds for this $a$, and $a\ge m$. Note that neither of these conditions involve the arbitrary estimator $\widehat{\beta}$. Thus PLS estimation under these conditions dominates all other estimators.

\section{Conclusions}

One purpose of this article has been to find arguments connected to the optimality of PLS type regression under certain conditions. As definite results on the properties of the PLS algorithm are lacking in the literature,  in my opinion all related results, either tentative or built upon exact calculations, are of interest. 

Another purpose of the article has been to illustrate how recent results from quantum theory can be seen in a statistical setting.

I conjecture that the results of this article can be generalized to the envelope model of Cook (2018). To formulate the precise theorems and to construct the proofs for this general case, however, are open problems.

As a concluding remark, I think that it may be of some value to use arguments from different scientific cultures in a theoretical statistical context. In general, communications across scientific borders may, as I see it, be valuable for progress in science.

Further applications of quantum theory to statistics are under investigation.

This has been a purely theoretical article. The result, however, that PLS regression has optimality properties under certain conditions, is of applied interest. PLS regression has a long range of applications in many fields; for an overview, see Mehmod and Ahmed (2015).

\section*{Acknowledgments}

I want to thank Soumitro Auddy, Barbara Bazzana, Magdy E. El-Adly, David J. Olive, Tejasvi Ravi, David Schneider, and Enrico Terzi for comments to an earlier version of this article. I am also grateful to Solve S\ae b\o \ and Trygve Alm\o y for discussions.

\section*{Appendix 1. Partial least squares regression.}

In this Appendix, take as a point of departure data $(\bm{x},y)$, where $\bm{x}$ is $p$-dimensional. For clarity I here include non-zero expectations $(\bm{\mu}_{x},\mu_{y})$.

Consider first the  population version of the well known PLS algorithm: 

Take
  $\bm{e}_{0}=\bm{x}-\bm{\mu}_{x}$, $f_{0}=y-\mu_{y}$, and for $j=1,2,...,m$ compute successively:
  \begin{equation} \label{wt}\bm{w}_{j}=\mathrm{cov}(\bm{e}_{j-1}, f_{j-1}),\ \ \
    t_{j}=\bm{w}_{j}'\bm{e}_{j-1},\end{equation}
  \begin{equation}\label{pq}\bm{p}_{j}=\mathrm{cov}(\bm{e}_{j-1}, t_{j})/\mathrm{var}(t_{j}),\ \
    q_{j}=\mathrm{cov}(f_{j-1},t_{j})/\mathrm{var}(t_{j}),\end{equation}
  \[\bm{e}_{j}=\bm{e}_{j-1}-\bm{p}_{j}t_{j},\ \ \
    f_{j}=f_{j-1}-q_{j}t_{j}.\]
    Go to the first step.
    
  It can be proved (Helland, 1990), and is important in this connection, that under the reduced model given by the hypothesis $H_m$, this algorithm stops automatically after $m$ steps when
  $m<p$: It stops because $\bm{w}_{m+1}=\mathrm{cov}(\bm{e}_{m},f_{m})=0$. After those $m$ steps we get the representations

  \begin{equation}
    \label{latent}
    \bm{x}=\bm{\mu}_{x}+\bm{p}_{1}t_{1}+...+\bm{p}_{m}t_{m}+\bm{e}_{m},\ \ y=\mu_{y}+q_{1}t_{1}+...+q_{m}t_{m}+ f_{m}
  \end{equation}

  with the corresponding PLS population prediction
 \[y_{m,PLS}=\mu_{y}+q_{1}t_{1}+...+q_{m}t_{m}=\mu_{y}+\bm{\beta}_{m,PLS}'(\bm{x}-\bm{\mu}_{x})= \mu_{y}+\bm{\beta}(\theta)'(\bm{x}-\bm{\mu}_{x})\]

In the ordimary PLS algorithm we have data on $n$ units $(\bm{X},\bm{y})$, where $\bm{X}$ is a $n\times p$ matrix. The algorithm runs as above, but with covariances and variances replaced by estimated covariances and variances. The algorithm will then in genaral not stop automatically; it can be run in $a$ steps, where $a$ usually is found by cross-validation or by using a test-set of data.

In Helland (1990) and elsewhere, the PLS algorithm is reformulated in terms of the Krylov sequence $\bm{\sigma}_{xy},{\mathit\mathbf{\Sigma}}_{xx} \bm{\sigma}_{xy},...,\mathit\mathbf{\Sigma}_{xx}^{m-1}\bm{\sigma}_{xy}$. The hypothesis $H_m$ can equivalently be written in the form that $\bm{\beta}$ is spnned by the $m$ vectors $\bm{e}_j;\  j=1,...,m$ in this Krylov sequence: $\bm{\beta}=\sum_{j=1}^m \alpha_j\bm{e}_j$.
The $a$ step PLS estimator can be written as $\widehat{\bm{\beta}_a}=\sum_{j=1}^a \widehat{\alpha_j}\widehat{\bm{e}_j}$, where $\{\widehat{\bm{e}_j}\}$ is the Krylov sequence with $a$ terms and with covariances and variances replaced by empirical covariances and variances, and where $\{\widehat{\alpha_j}\}$ are suitable coefficients..

Among statistical articles and books that have discussed various aspects of the PLS algorithm, I can mention Frank and Friedman (1993), Garthwaite (1994), Stoica and Sr{\"o}derstr{\"o}m (1998), Sundberg (1999), Kr{\"a}mer and Sugiyama (2012), Foschi (2015), Cook and Forzani (2019), Cook and Forzani (2024) and Olive and Zhang (2024).

\section*{Appendix 2. A quantum-mechanical supplement to Section 6.}

It is possible, and convenient, to approximate continuous parameters by parameters assuming a fonite number of values, for a physical example, see Subsection 5.3 in Helland (2021).

Connected to a discrete maximal accessible variable $\eta_t$ with operator $A^{\eta_t} =\sum_{i=1}^r \eta_{it} \bm{u}_i \bm{u}_i^\dagger$, there is a density operator $\rho^{\eta_t} = \sum_{i=1}^r p_i \bm{u}_i \bm{u}_i^\dagger$, which expresses the knowledge that the statistician $B$ might have. Here, the probabilities $p_i$ may be either prior or posterior Bayesian probabilities, or based upon a confidence distribution, see Schweder and Hjort (2016).  For another discrete maximal accessible variable $\theta_t$. one variant of Born's formula is then
\begin{equation}
\mathrm{E}_{\theta_t} [\xi(\theta_t)|\rho^{\eta_t} ] = \mathrm{trace}(\rho^{\eta_t} \xi(A^{\theta_t} )),
\label{80}
\end{equation}
where $\xi(A^{\theta_t})=\sum_{i=1}^{r^t}\xi(\theta_{it}) \bm{v}_i\bm{v}_i^\dagger$ for the basis $\{\bm{v}_i\}$ connected to the operator $A^{\theta_t}$.

A non-informative probability distribution of $\eta_t$ can be expressed by $\rho^{\eta_t} = r^{-1} I$, and gives
\begin{equation}
\mathrm{E}_{\theta_t} [\xi(\theta_t )|\rho^{\eta_t} ] =r^{-1} \mathrm{trace}( \xi(A^{\theta_t} ))= r^{-1}\sum_{i=1}^r \xi(\theta_{it} )= \overline{\xi(\theta_t )}.
\label{81}
\end{equation}

Use this for $\xi_j(\theta)= (\gamma_j -\mu_j )^2$, discretize $\theta$ as $\theta_t$ and llet $t\rightarrow\infty$. Then the mean converges to the local expectation, and we find that
\begin{equation}
\mathrm{E}_{\theta}[\sum_{j=1}^m \lambda_j (\gamma_j -\mu_j )^2 |\rho^\eta]=\sum_{j=1}^m \lambda_j  \mathrm{E}_{\gamma_j}(\gamma_j-\mu_j)^2,
\label{81}
\end{equation}
where $\rho^{\eta}=\mathrm{lim}_{t\rightarrow\infty} \rho^{\eta_t}$ is again a non-informative density operator.

\end{spacing}

\end{document}